%% file: 2-nearly_Platonic_07e.tex
\documentclass[10pt]{article}
\usepackage{amsfonts}
\usepackage{amsthm}
\usepackage{amssymb}
\usepackage{graphicx, enumerate}
\usepackage{amsmath}
\usepackage{latexsym}
\usepackage{longtable}
\usepackage{tabularx}
\usepackage{amsmath}
\usepackage{amsfonts}
\usepackage{amsthm}
\usepackage{setspace}
\usepackage{graphicx}
\usepackage{float}
\usepackage{rotating}
\usepackage{tikz}
\usepackage{verbatim}
\usepackage{subfig}

\newtheorem{theorem}{Theorem}[section]
\newtheorem{lemma}[theorem]{Lemma}
\newtheorem{prop}[theorem]{Proposition}

\newtheorem{obs}[theorem]{Observation}

\newtheorem*{rem*}{Remark}

\theoremstyle{definition}
\newtheorem{definition}[theorem]{Definition}

\newcommand{\dist}{{\rm{dist}}}

\vspace{5cm}

\begin{document}

\textwidth4.5true in
\textheight7.2true in

\title
{\Large \sc \bf {$2$-nearly Platonic graphs are unique v07e}}
\date{}
\author{{{D. Froncek, M.R. Khorsandi, S.R. Musawi, J. Qiu}}}
\maketitle

\begin{abstract}
A $2$-nearly Platonic graph of type $(k|d)$ is a $k$-regular planar graph with $f$ faces, $f-2$ of which are of degree $d$ and the remaining two are of degrees $m_1,m_2$, both different from $d$. Such a graph is called balanced if $m_1=m_2$. We show that  all $2$-nearly Platonic graphs are necessarily balanced. This proves a recent conjecture  by Keith, Froncek, and Kreher.
\end{abstract}

\noindent
\textbf{Keywords:}  Planar graphs, regular graphs, Platonic graphs

\noindent
\textbf{Mathematics Subject Classification:} 05C10

\input{intro_07e}

\input{notation_07a}

\input{known_07a}

\input{touching_07e}

\input{non-touching_07e}

\input{conclusion_07e}

\input{references_07c}

\end{document}

%% file: intro_07e.tex
\section{Introduction}\label{sec:intro}

Throughout this paper, all graphs we consider are finite, simple, connected, planar, undirected and non-trivial graph. 

A graph is said to be planar, or embeddable in the plane, if it can be drawn in the plane such that each common point  of two edges is a vertex. This drawing of a planar graph $G$ is called a planar embedding of $G$ and can itself be regarded as a graph isomorphic to $G$. Sometimes, we call a planar embedding of a graph a \emph{plane graph}. By this definition, it is clear that we need some matters of the topology of the plane. Immediately, after deleting the points of a plane graph from the plane, we have some maximal open sets (or regions) of points in the plane called  \emph{faces} of the plane graph. There exist exactly one unbounded region that we call it an \emph{outerface} of the plane graph and other faces are called as \emph{internal faces}. 
The boundary of a face is the set of points of vertices and  edges touching the face. In the graph-theoretic language, the boundary of a face is a closed walk. A face is said to be incident with the vertices and edges in its boundary, and two faces are adjacent if their boundaries have an edge in common. 

A graph $G$ is  \emph{$k$-regular} when the degrees of all vertices are equal to $k$. A regular graph is one that is $k$-regular for some $k$. Let $G=(V,E,F)$ be a graph with the vertex set $V$, edge set $E$, and face set $F$. The well-known \textit{Euler's formula} states that if $G$ is a connected planar graph, then 
$$
|V| - |E| + |F| = 2.
$$

The \emph{length} of a face in a plane graph $G$ is the total length of the closed walk bounding the face in $G$. A cut-edge belongs to the boundary of only one face, and it contributes twice to its length.

A graph $G$ is  \emph{$k$-connected} if $|V(G)|>k$ and $G-X$ is connected for every set $X\subseteq V(G)$ with $|X|<k$.
If $|V(G)|>1$ and $G-A$ is connected for every set $A\subseteq E(G)$  of fewer than $\ell$ edges, then $G$ is called \emph{$\ell$-edge-connected}.


Platonic solids are a well-known family of five three-dimensional polyhedra. There is no reliable information about their first mention, and different opinions have been taken \cite{Atiyah-Sutcliffe,Lloyd}. However, they are attractive for mathematicians and others, in terms of some symmetries that they have. In the last two centuries, many of authors have paid attention to the polyhedra and they have extended it to convex and concave polytopes, see, e.g.,  \cite{Grunbaum}.

But what matters from the combinatorial point of view is that a convex polyhedron can be embedded on a sphere, and then we can map it on a plane so that the images of lines on the sphere do not cut each other in the plane. In this way, we have corresponded a polyhedron on the sphere with a planar graph in the plane. 
Steinitz's theorem (see, e.g., \cite{Grunbaum}) states that a graph $G$ with at least four vertices is the network of vertices and edges of a convex polyhedron if and only if $G$ is planar and $3$-connected.  

In 1967, Gr\"unbaum considered 3-regular and connected planar graphs and he got some results. For example, for a $3$-regular connected planar graph and $k\in\{2,3,4,5\}$,  it is proved that if the length of all faces but $t$ faces is divisible by $k$ then $t\ge2$ and if $t=2$ then two exceptional faces have not a common vertex \cite{Grunbaum}. In 1968, in his Ph.D thesis, Malkevitch proved the same results for 4 and 5-regular 3-connected planar graphs \cite{Malkevitch-2}.   
Several papers are devoted to the study of this topic, but all of them have considered the planar graphs such that the lengths of all faces but some exceptional faces are a multiple of $k$ and $k\in \{2,3,4,5\}$ (see \cite{Crowe,Hornak-Jendrol,HorJuc,Jendrol,Jendrol-Jucovic}).

A $k$-regular planar graph with $f$ faces is a \emph{$t$-nearly Platonic graph} of type $(k|d)$ if $f>2t, f-t$ of its faces are of size $d$ and the remaining $t$ faces are of sizes other than $d$. The faces of size $d$ are often called \emph{common faces}, and the remaining ones \emph{exceptional} or \emph{disparate} faces. When $t\geq2$ and all disparate faces are of the same size, then the graph is called a \emph{balanced $t$-nearly Platonic graph}.

Keith, Froncek, and Kreher~\cite{KFK1,KFK2} and Froncek, Khorsandi, Musawi, and Qiu~\cite{FKQ-1nP} proved recently that there are no $1$-nearly Platonic graphs.

Deza, Dutour Sikiri\v{c}, and Shtogrin~\cite{DSS}  classified for each admissible pair $(k|d)$  all possible sizes of the exceptional faces of balanced $3$-nearly Platonic graphs and sketched a proof of the completeness of the list. Froncek and Qiu~\cite{FKQ-3nP} provided a detailed combinatorial proof of existence of infinite families of such graphs for each of the listed exceptional sizes.

There are 14 well-known families of balanced $2$-nearly Platonic graphs (see, e.g.,~\cite{DSS} or~\cite{KFK1}). Deza, Dutour Sikiri\v{c}, and Shtogrin~\cite{DSS} provide a list and offer a sketch of a proof of the completeness of the list.  Keith, Froncek, and Kreher conjectured~\cite{KFK1} that every $2$-nearly Platonic graph must be balanced. 

We show that the only admissible types of $2$-nearly Platonic graphs are $(3|3),(3|4),(3|5),(4|3)$, and $(5|3)$, and that all $2$-nearly Platonic graphs are balanced. We also prove in detail that the list of 14 families presented by Deza, Dutour Sikiri\v{c}, and Shtogrin~\cite{DSS} is complete.

%% file: notation_07a.tex
\section{Terminology and notation}\label{sec:notation}

We say that the two exceptional faces are \emph{touching each other} or simply \emph{touching}, if they share at least one vertex. Similarly, an exceptional face will be called \emph{self-touching} if a vertex appears on the boundary of the face more than once.

\begin{definition}
	Let $G$ be a graph and $S\subseteq V(G)$. Then $\langle S\rangle$,  the induced subgraph by $S$, denotes the graph on $S$ whose edges are precisely the edges of $G$ with both ends in $S$. Also, $G-S$ is obtained from $G$ by deleting all the vertices in $S$ and their incident edges. If $S=\{x\}$ is a singleton, then we write $G-x$ rather than $G-\{x\}$.
\end{definition}

\begin{definition}\label{def:block}
	A $(k;k_1,k_2|d)$-block $B$ of order $n$ is a $2$-connected planar graph with $n-2$ vertices of degree $k$, two vertices $x$ and $y$ with $\deg(x)=k_1,\deg(y)=k_2$ where $2\leq k_2\leq k_1 <k$, all faces but one of degree $d$, and the remaining  face of degree $h\neq d$, where vertices $x,y$ belong to the face of degree $h$. 
\end{definition}

\begin{definition}\label{def:endblock}
	A $(k;k_1|d)$-endblock of order $n$ is a $2$-connected planar graph with $n-1$ vertices of degree $k$, one vertex $x$ with $\deg(x)=k_1$ where $1< k_1<k$, all faces but one of degree $d$, and the remaining  face of degree $h\neq d$, where the vertex $x$ belongs to the face of degree $h$.
\end{definition}

When we speak about blocks, we always assume that the exceptional face is the outer face. Let the boundary path of the exceptional face of $(k;k_1,k_2|d)$-block $B$  be of length $h=a+b$. We denote it $x=x_0,x_1,\dots,x_{a}=y,x_{a+1},\dots,$ $x_{a+b-1}$ and always assume that $a\leq b$. When we need to specify $a$ and $b$ in our arguments, we denote such a block as $(k;k_1,k_2|d ,\langle a,b\rangle)$-block.

Similarly, if we need to specify the length $h$ of the exceptional face in an endblock, we will denote it as $(k;k_1|d,\langle h\rangle)$-endblock.

We observe that when the exceptional faces touch in exactly one vertex, say $z$, then by splitting $z$ into two vertices $x$ and $y$ we obtain a $(k;k_1,k_2|d,\langle a,b\rangle)$-block, where $\deg(x)\geq2,\deg(y)\geq2$ and $a$ and $b$ are the sizes of the two original exceptional faces, $k_2=2$ and $2\leq k_1\leq 3$. Consequently, such a graph would have to be of type $(4|3)$ or $(5|3)$.

%% file: known_07a.tex
\section{Related results}\label{sec:known}

We will use the following results in our proofs.

\begin{theorem}\label{thm:2-conn-edges-on-bound} 
	In a $2$-connected plane graph, all facial boundaries are cycles and each edge lies on the boundary of two distinct faces. 
\end{theorem}

\begin{theorem}[\cite{KFK2}]\label{thm:no-2-conn-1NP}
	There are no $2$-connected $1$-nearly Platonic graphs.
\end{theorem}

The following result is stated in~\cite{DS}, and a detailed proof is given in~\cite{FKQ-1nP}.

\begin{theorem}[\cite{DS},\cite{FKQ-1nP}]\label{thm:no-endblocks}
	There are no endblocks of any admissible type.
\end{theorem}

The non-existence of 1-nearly Platonic graphs with connectivity 1 follows directly from the above Theorem.

\begin{theorem}[\cite{FKQ-1nP}]\label{thm:no-1-conn-1NP}
	There are no $1$-nearly Platonic graphs with connectivity $1$.
\end{theorem}

%% file: touching_07e.tex
\section{Touching exceptional faces}\label{sec:touching}

\subsection{No self-touching exceptional face}\label{subs:n-self-touch}

First we observe that if there are two touching exceptional faces, each of them must be touching the other but not itself.

\begin{lemma}\label{lem:no-self-touch}
There is no self-touching exceptional face in any connected $2$-nearly Platonic graph.
\end{lemma}

\begin{proof}
	Suppose that in a connected $2$-nearly Platonic graph $G$ there is a self-touching exceptional face. Then $G$ has a cut-vertex. It is well known that in such a case there is a block $B$ containing exactly one cut-vertex. (The term \emph{block} is here used in the usual sense, that is, for a maximal $2$-connected subgraph.) Because we have $\delta(G)\geq 3$, block $B$ must be an endblock as defined above, which does not exist by Theorem~\ref{thm:no-endblocks}, which proves the claim.
\end{proof}

Recall that if the exceptional faces are touching at exactly one vertex, then the graph must be of type $(4|3)$ or $(5|3)$.

\subsection{Excluding non-admissible values of $a$}\label{subs:non-adm-a}

Now we reduce the family of blocks that we need to investigate just to the cases where $a\leq d$.

\begin{prop}\label{prop:b-reduction}
	 If there exists a $(k;k-1,k_2|d, \langle a,b\rangle)$-block $B$ with $a>d$, then there exists a $(k;k-1,k_2|d, \langle a-d,b+d\rangle)$-block $B'$.
\end{prop}

\begin{proof}
	Take the block $B$, remove edge $x_{d-1}x_{d}$ and replace it by edge $x_0x_{d-1}$. Vertex $x_0$ is now of degree $k$ while $x_{d}$ is of degree $k-1$. The lengths of the boundary segments are now $a-d$ and $b+d$, respectively.
	
	If the graph is no longer 2-connected, then there is a cut-vertex $x_j$ with $j>a$.  The subgraph bounded by $x_0,x_{d-1},x_j,x_{j+1},\dots,x_{a+b-1}$ is then an endblock of type $(k;k_1|d)$ with $2\leq k_1\leq k-1$ and exceptional vertex $x_j$, which does not exist by Theorem~\ref{thm:no-endblocks}. 
	
	Therefore, the graph is still 2-connected, and the conclusion follows.
\end{proof}

Hence, from now on we can only consider  $(k;k-1,k_2|d)$-blocks with $1\leq a\leq d$. First we exclude the existence of $(k;k-1,k_2|d)$-blocks with $a=d$.

\begin{lemma}\label{lem:(k,d)_no_a=d}
	There is no $(k;k-1,k_2|d)$-block with $a=d$ for any admissible $k$.
\end{lemma}

\begin{proof}
	Suppose such a block exists. 	First assume $k_2=2$. We remove the edge $x_{d-1}x_d=x_{d-1}y$ and replace it by the edge $x_0x_{d-1}=xx_{d-1}$. This way we obtain a new internal face of size $d$. Vertex $x_0$ is now of degree $k$, vertex $x_d=y$ is of degree one and all other vertices are of degree $k$.
	
	Then we remove the vertex $y=x_d$, and vertex $x_{d+1}$ is now of degree $k-1$. But then we have a $(k;k-1|d)$-endblock, which does not exist by Theorem~\ref{thm:no-endblocks}.

	If $k_2\geq3$, then $4\leq k\leq 5$ and we must have $d=a=3$.	Remove the edge $x_{2}x_3=x_{2}y$ and replace it by the edge $x_0x_{2}=xx_2$. This creates a new internal triangular face. Vertex $x_0$ is now of degree $k$, vertex $x_3=y$ is of degree $k_2-1\geq2$ and all other vertices are of degree $k$.

	Similarly as in Proposition~\ref{prop:b-reduction}, if the graph now has a cut vertex $x_j$ originally belonging to the triangle $x_{2},x_3,x_j$, then the graph bounded by the cycle $x_0,x_2,x_j,x_{j+1},\dots,x_{3+b-1}$ is  an endblock of type $(k,d;k_1)$ with $2\leq k_1\leq k-1$, which does not exist by Theorem~\ref{thm:no-endblocks}. 
	
	So the graph must be still 2-connected. The edge $x_2x_3$ must have belonged to a triangle $x_2,x_3,z$ and  we have a $(k;k_2-1|d)$-endblock with boundary $x_0,x_2,z,x_3,\dots,x_{a+b-1}$ and $x_3$ of degree $k_2-1\geq 2$, which does not exist by Theorem~\ref{thm:no-endblocks}.  
\end{proof}

The case of $a=d-1$ can be easily excluded for $k_2<k-1$.

\begin{lemma}\label{lem:(k,d)_no_a=d-1}
	There is no $(k;k-1,k_2|d)$-block with $a=d-1$ and $k_2<k-1$ for any admissible $k$.
\end{lemma}

\begin{proof}
	If such a block exists, then by adding edge $xy$ we create a new internal face of size $d$. Vertex $y$ is still of degree less than $k$ while all other vertices are of degree $k$. This new graphs would be a $(k;k_2|d)$-endblock, which does not exist by Theorem~\ref{thm:no-endblocks}.  
\end{proof}

Now we exclude the existence of $(k;k-1,k_2|d)$-blocks with $a=1$.

\begin{lemma}\label{lem:(k,d)_no_a=1}
	There is no $(k;k-1,k-1|d)$-block with $a=1$ for any admissible $k$.
\end{lemma}

\begin{proof}
	Suppose such a block $B$ exists. If $d=3$, then add a new vertex $z$ and edges $xz$ and $yz$. This creates a $(k,2|3)$-endblock, which does not exist by Theorem~\ref{thm:no-endblocks}.
	
	For $d=4$, create a copy $B'$ of $B$ with vertices $x'$ and $y'$ corresponding to $x$ and $y$, respectively. Then add edges $xx'$ and$yy'$ to create a new internal face of size four. All vertices in this new graph are of degree $k$, and the outer face is of size at least six. The resulting graph would now be  2-connected and 1-nearly Platonic, but such a graph does not exist by Theorem~\ref{thm:no-2-conn-1NP}.
	
	For $d=5$, again create $B'$ as above, and an extra vertex $z$. Add edges $xx', yz, y'z$ to obtain a new internal face $x,x',y',z,y$ of size five, vertex $z$ of degree two, and all other vertices of degree $k$. The new graph now would be a $(k;2|5)$-endblock, which does not exist by Theorem~\ref{thm:no-endblocks}.

\end{proof}

For $k=5$, we can even exclude large values of $a$ even for $k_1=k-2=3$.

\begin{lemma}\label{lem:(5,3;3,d_2)-reduction}
	If there is a $(5;3,d_2|3,\langle a,b\rangle)$-block $B$ with $a>3$, then there is a $(5;3,d_2|3,\langle a-3,b+3\rangle)$-block. Consequently, there exists a $(5;3,d_2|3,\langle a',b'\rangle)$-block $B'$  with $1\leq a'\leq3$.
\end{lemma}

\begin{proof}
	As always, $b\geq a$ by our assumptions above. 
	
	First, call $z$ the common neighbor of $x_2$ and $x_3$. Take the edge $x_2x_3$ and replace it by edge $x_0x_2$. We have a new triangular face $x_0,x_1,x_2$ and 
	$\deg(x_0)=\deg(x_3)=4$ and $\deg(x_a)=d_2$. 
	Now take the edge $zx_3$ and replace it by edge $x_0z$.
	We have a new triangular face $x_0,x_2,z$ and 
	$\deg(x_0)=5,\deg(x_3)=3$ and $\deg(x_a)=d_2$. 
	If the original triangular face $x_2,x_3,z$ had $z=x_j$ for some $j>a$, then the cycle 		
	$x_0,z=x_j,x_{j+1},\dots,x_{a+b-1}$ 
	is now bounding a $(5;3|d')$-endblock for some $d'\leq 4$, which is impossible by Theorem~\ref{thm:no-endblocks}.
	
	Hence, $z$ is an inner vertex of the block $B$. But in this case we have obtained another $(5;3,d_2|d,\langle a-3,b+3\rangle)$. If $a-3\leq3$, we are done. Otherwise, we repeat the reduction until we arrive at a block $(5;3,d_2|d,\langle a',b'\rangle)$ with $a'=a-3t\leq3$ as desired.	 
\end{proof}

\subsection{Uniqueness of $(3;2,2|d)$-blocks}\label{subs:(3;2,2|d)-unique}

	We can now show uniqueness of the $(k;k-1,k-1|d)$-blocks for all $d=3,4,5$. 

\begin{lemma}\label{lem:(3,d;2,2)-uniq}
	The $(3;2,2|d)$-blocks are unique for each $d=3,4,5$. 
\end{lemma}

\begin{proof}
\textbf{Case $d=3$}	

	By Proposition~\ref{prop:b-reduction} and Lemmas~\ref{lem:(k,d)_no_a=d} and~\ref{lem:(k,d)_no_a=1}, we only need to investigate the case $a=2$.

	If $a=2$, then we can add the edge $xy=x_0x_2$ to obtain a $3$-regular $2$-connected graph with all faces size $d=3$, except  possibly the outer one. If the new outer face is of size greater than three, then we have obtained a 2-connected  1-nearly Platonic graph.  By Theorem~\ref{thm:no-2-conn-1NP}, there is no such graph, hence the outer face must be a triangle $x_0,x_2,x_3$ and the new graph is the tetrahedron. Thus, the original $(3;2,2|3)$-block was the tetrahedron without one edge shown in Figure~\ref{fig:(3;2,2|3)}.

	\begin{figure}[H]
		\centering
		\includegraphics[scale=.08]{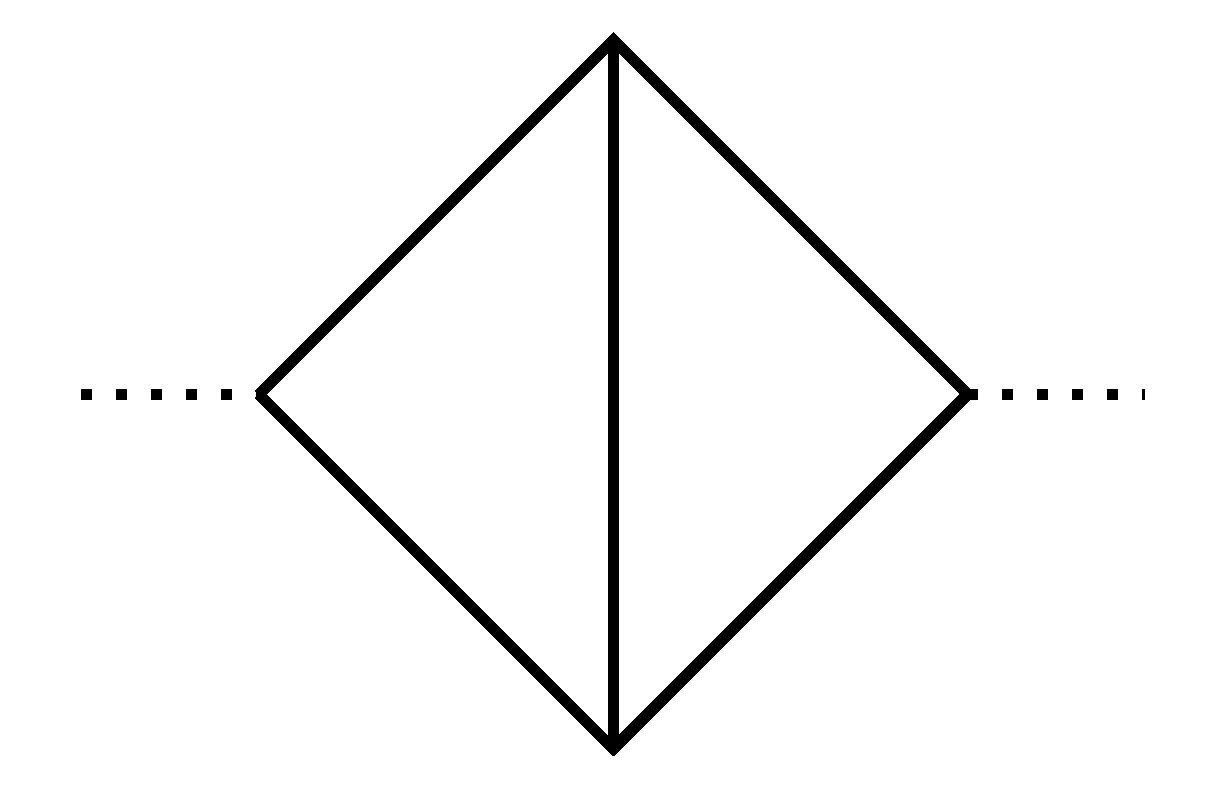}
		\caption{Unique $(3;2,2|3)$-block}\label{fig:(3;2,2|3)}
	\end{figure}

\vskip6pt
\noindent	
\textbf{Case $d=4$}	

	Similarly as above, by Proposition~\ref{prop:b-reduction} and Lemmas~\ref{lem:(k,d)_no_a=d} and~\ref{lem:(k,d)_no_a=1}, we only need to consider $2\leq a\leq 3$. 

	If $a=2$, by adding a new vertex $z$ and edges $x_0z$ and $zx_2$, both $x_0$ and $x_2$ now have degree 3, and we obtain a $(3;2|4)$-endblock, which does not exist by Theorem~\ref{thm:no-endblocks}.
	
	When $a=3$, then by adding the edge $xy=x_0x_3$ we obtain a $3$-regular graph $2$-connected graph with all faces except possibly the outer one of size $d=4$. By Theorem~\ref{thm:no-2-conn-1NP}, there is no 2-connected 1-nearly Platonic graph, hence the outer face must be a 4-cycle $x_0,x_3,x_4,x_5$ and the new graph is the cube. Thus, the original $(3;2,2|4)$-block was the cube without one edge shown in Figure~\ref{fig:(3;2,2|4)}.

	\begin{figure}[H]
		\centering
		\includegraphics[scale=.15]{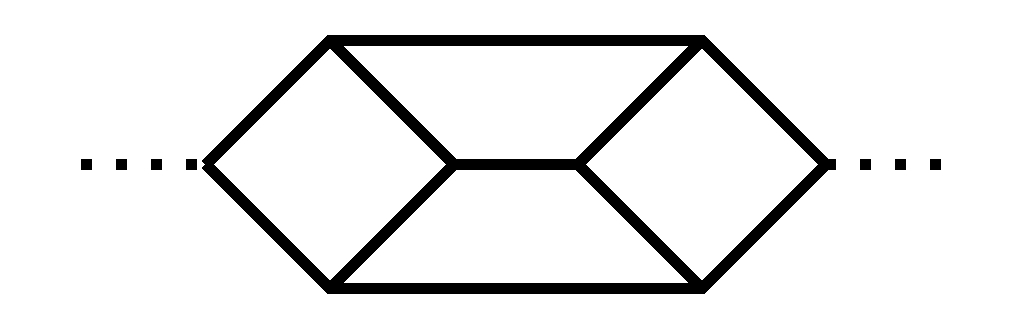}
		\caption{Unique $(3;2,2|4)$-block}\label{fig:(3;2,2|4)}
	\end{figure}

\vskip6pt
\noindent
\textbf{Case $d=5$}	

	We must consider only $2\leq a\leq 4$ as proved in Proposition~\ref{prop:b-reduction} and Lemmas~\ref{lem:(k,d)_no_a=d} and~\ref{lem:(k,d)_no_a=1}.

	For $a=2$, we take two copies of $B$ and add two new vertices $z_1,z_2$ and edges
	$z_1z_2,xz_1,yz_2,x'z_1,y'z_2$. This creates two new faces of size five, bounded by  cycles 
	$x=x_0,x_1,x_2=y,z_2,z_1$ and $x'=x'_0,x'_1,x'_2=y',z_2,z_1$. The outer face of this new amalgamated graph is of size at least eight, which is impossible, since the graph would be 2-connected and 1-nearly Platonic, which is impossible by Theorem~\ref{thm:no-2-conn-1NP}.
	
	For $a=3$, adding a new vertex $z$ and edges $xz=x_0z$ and $zy=zx_3$ we would obtain a $(3;2|5)$-endblock, which does not exist by Theorem~\ref{thm:no-endblocks}. 
	
	Finally, when $a=4$, by adding the edge $xy=x_0x_4$ we get a new face of size five bounded by $x_0,x_1,x_2,x_3,x_4$ and all vertices are now of degree three. By Theorem~\ref{thm:no-2-conn-1NP}, there are no 2-connected 1-nearly Platonic graphs. Since our new graph is 2-connected, the new outer face must be a pentagon as well. Thus, the new graph is a dodecahedron and the original graph was a dodecahedron without one edge shown in Figure~\ref{fig:(3;2,2|5)}.
\end{proof}

	\begin{figure}[H]
		\centering
		\includegraphics[scale=.12]{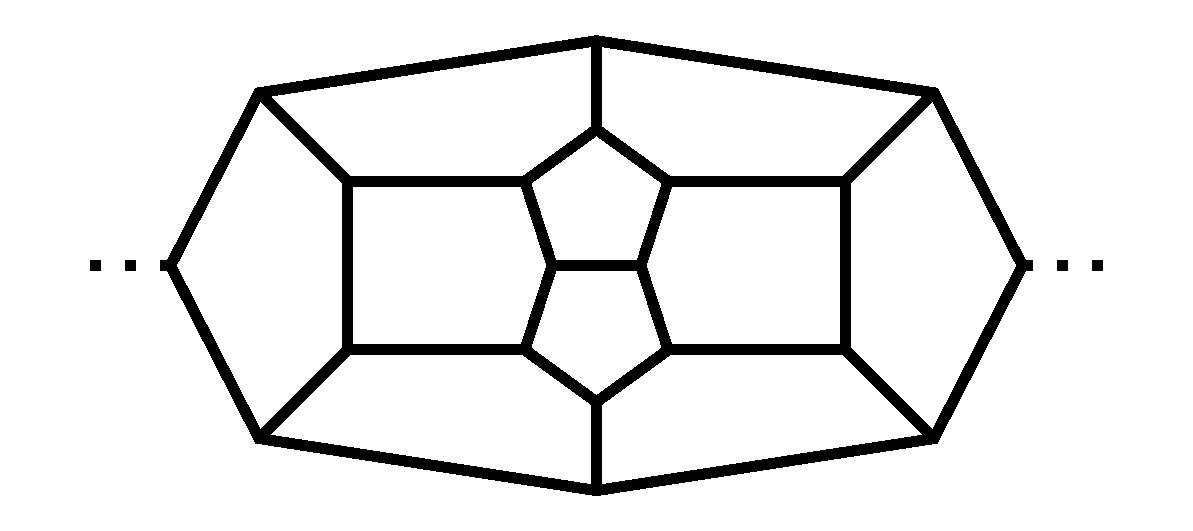}
		\caption{Unique $(3;2,2|5)$-block}\label{fig:(3;2,2|5)}
	\end{figure}

\subsection{Uniqueness of $(4;k_1,k_2|d)$-blocks}\label{subs:(4;k_1,k_2|d)-unique}

Now we discuss existence of $(4;k_1,k_2|d)$-blocks. The only admissible value of $d$ is $d=3$, and the only possibilities are $(4;3,3|3)$-blocks, $(4;3,2|3)$-blocks, and $(4;2,2|3)$-blocks.

\begin{lemma}\label{lem:(4,3;3,3)-uniq}
	The $(4;3,3|3)$-block is unique. 
\end{lemma}

\begin{proof}
	We must consider only $a=2$.  
	
	We again add to $B$ the edge $xy=x_0x_2$ similarly as in the case of $k=3$ and obtain a $4$-regular  $2$-connected graph with all internal faces of size $d=3$. By Theorem~\ref{thm:no-2-conn-1NP}, there is no 2-connected 1-nearly Platonic graph, hence the outer face must be a triangle $x_0,x_2,x_3$ and the new graph is the octahedron. Thus, the original $(4;3,3|3)$-block was the octahedron without one edge shown in Figure~\ref{fig:(4;3,3|3)}.
\end{proof}

\begin{figure}[H]
	\centering
	\includegraphics[scale=.15]{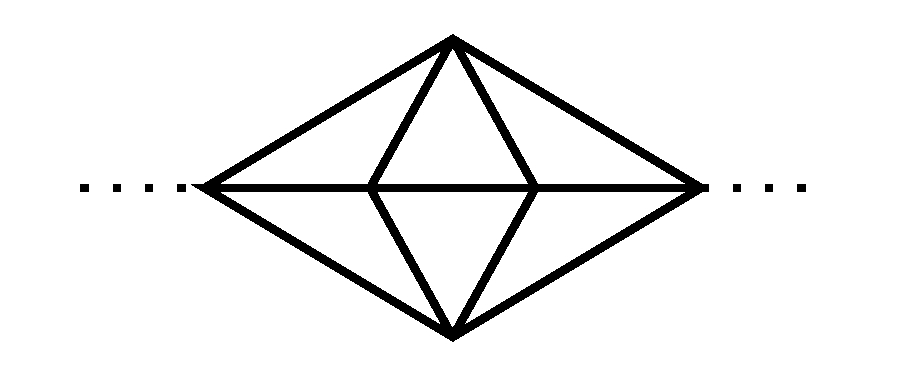}
	\caption{Unique $(4;3,3|3)$-block}\label{fig:(4;3,3|3)}
\end{figure}

The following result is a direct corollary of Lemmas~\ref{lem:(k,d)_no_a=d},~\ref{lem:(k,d)_no_a=d-1}, and~\ref{lem:(k,d)_no_a=1}.

\begin{lemma}\label{lem:(4,3;3,2)-DNE}
	There is no $(4;3,2|3)$-block. 
\end{lemma}

\begin{proof}
	First assume such a block $B$ exists for $a=1$. Then we take two copies of $B$, say $B$ and $B'$ and amalgamate vertices $y$ and $y'$, obtaining another vertex of degree two and add edge $xx'$. But then we have constructed a 2-connected 1-nearly Platonic graph of type $(4|3)$, which does not exist by Theorem~\ref{thm:no-2-conn-1NP}.
	
	We cannot have  $a=2$ by Lemma~\ref{lem:(k,d)_no_a=d-1},  or  $a=3$ by Lemma~\ref{lem:(k,d)_no_a=d}. Hence, the proof is complete.
\end{proof}

The only remaining case for 4-regular blocks is more complex.

\begin{lemma}\label{lem:(4,3;2,2)-uniq}
	The $(4;2,2|3)$-block is unique. 
\end{lemma}

\begin{proof}
	Let such a block be called $B$. If $a=1$, we create three copies $B^0,B^1,B^2$ of $B$ with the vertices of degree two denoted $x^i$ and $y^i$ in each copy $B^i$. Assume that $B$ has $t$ internal triangular faces and observe that $b>1$ . Then we amalgamate $x^i$ with $y^{i+1}$ for all $i=0,1,2$, where the superscripts are calculated modulo 3. This way we obtain a 2-connected 4-regular graph with $3t+1$ inner triangular faces and the outer face of size $3b\geq6$. Because no such graph exists by Theorem~\ref{thm:no-2-conn-1NP}, this case is impossible.
		
	When $a=2$, then we add the edge $xy=x_0x_2$ and obtain a $(4;3,3|3)$-block with the vertices of degree three joined by an edge, which cannot exist by Lemma~\ref{lem:(k,d)_no_a=1}. Hence, $a\neq 2$.
	
	For $a=3$ and $b=3$, the boundary is the 6-cycle $x_0,x_1,\dots,x_5$ with $\deg(x_0)=\deg(x_3)=2$. Therefore, inside the 4-cycle $x_1,x_2,x_4,x_5$ with edges $x_1x_2,x_2x_4,x_4x_5,x_5x_1$ there must be a vertex $v$, adjacent to all vertices of the cycle.  This gives the  unique $(4;2,2|3)$-block shown in Figure~\ref{fig:(4;2,2|3)}. Notice that amalgamating $x_0$ with $x_3$ produces an octahedron.

	\begin{figure}[H]
		\centering
		\includegraphics[scale=.17]{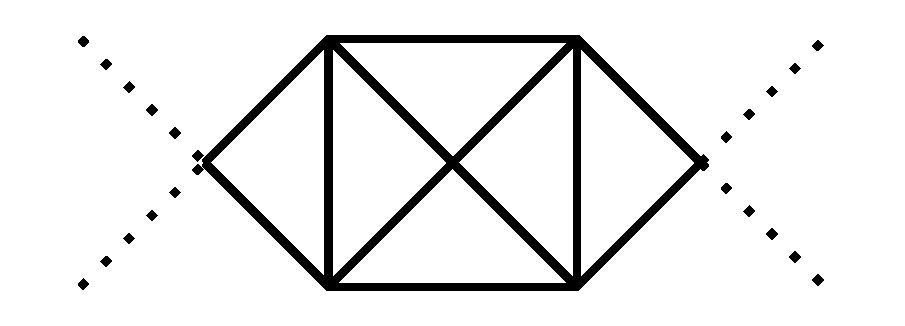}
		\caption{Unique $(4;2,2|3)$-block}\label{fig:(4;2,2|3)}
	\end{figure}

	Now we need to show that when $a=3$, we cannot have $b>3$. Suppose we can. But then by amalgamating $x_0$ with $x_3$ as above into a vertex $x'$ we obtain one new inner triangular face $x',x_1,x_2$ and an outer face $x_0,x', x_3,\dots, x_{3+b-1},x_0$ of size $b+1$. Because $b>3$, the outer face is of size at least four, and we have a 2-connected 1-nearly Platonic graph of type $(4|3)$. No such graph exists by Theorem~\ref{thm:no-2-conn-1NP}, which implies $b=3$, contradicting our assumption that $b>3$.

	Finally, let $a>3$. 
	Let $a=3c+r$ for some $c>1$ and $0\leq r\leq2$. Remove edges $x_2x_3,\dots,x_{3c-1}x_{3c}$ and replace them by edges $x_0x_2,\dots,x_{3c-1},x_{3c-3}$.

	Let $i$ be the smallest subscript such that the edge $x_{3i-1}x_{3i}$ belonged to a triangle $x_{3i-1},x_{3i},x_j$ for some $j>a$ and the previous edges (if any) $x_{3s-1}x_{3s}$ belonged to triangles $x_{3s-1},x_{3s},z_{3s}$, where $z_{3s}$ is not a boundary vertex. Then the graph bounded by the cycle $x_0,x_2,z_3,x_3,x_5,\dots,x_{3i-1},x_j,x_{j+1},\dots,x_{a+b-1}$ has $x_0$ of degree 3 and $x_j$ of degree 2 or 3. However, if $\deg(x_j)=2$, no such graph can exist by Proposition~\ref{prop:b-reduction} and Lemma~\ref{lem:(k,d)_no_a=d}. 
	
	If $\deg(x_j)=3$, the graph $B^*$ bounded by
		 $x_0,x_2,z_3,x_3,x_5,\dots,x_{3i-1},x_j$, 
		 $x_{j+1},\dots,x_{a+b-1}$ 
	would be a $(4;3,3|3)$-block. However, such a block is unique with $a=2$, while here we have $a^*\geq3$, because of the path 
			$x_0,x_2,z_3,x_3,x_5,\dots$,
			$x_{3i-1},x_j$, 
	where $i\geq1$. Therefore, this possibility can be ruled out as well.

	Thus, no edge $x_{3s-1}x_{3s}$ belongs to a triangle $x_{3i-1},x_{3i},x_j$. 
	
	Now if $a=3c$, after performing the edge operations above, we are left with $x_{a}$ having degree one and its only remaining neighbor is $x_{a+1}$. Removing $x_a$ we obtain again a $(4;3,3|3)$-block as in the previous paragraph, and the same contradiction.

	When $a=3c+1$, we end up with $\deg(x_0)=\deg(x_{a-1})=3$ and $\deg(x_{a})=2$. We create two copies of the block, say $B$ and $B'$, amalgamate $x_{a}$ with $x'_{a}$ and add a new edge $x_{a-1} x'_{a-1}$. This way we obtain a $(4;3,3|3,\langle 2a-1,2b\rangle)$-block. Since $2b>2a-1\geq 7$, no such block can exist by Lemma~\ref{lem:(4,3;3,3)-uniq}.
	
	Finally, for $a=3c+2$ we transform the graph so that  $\deg(x_0)=	\deg(x_{a-2})=3$ and $\deg(x_{a})=2$.  By adding the edge $x_{a-2} x_{a}$ we create a new inner triangle $x_{a-2},x_{a-1}, x_{a}$ and $\deg(x_{a-2})=4$ and $\deg(x_{a})=3$. The resulting graph is a $(4;3,3|3,\langle a-1,b\rangle)$-block.. But $b>a-1\geq 4$, and no such block can again exist by Lemma~\ref{lem:(4,3;3,3)-uniq}.
\end{proof}

\subsection{Uniqueness of $(5;k_1,k_2|d)$-blocks}\label{subs:(5;k_1,k_2|d)-unique}

Again, we must have $d=3$. Hence, we are left with blocks of type $(5;k_1,k_2|3)$ for $4\geq k_1\geq k_2\geq2$.

\begin{lemma}\label{lem:(5,3;4,4)-uniq}
	The $(5;4,4|3)$-block is unique. 
\end{lemma}

\begin{proof}
	It follows from Lemma~\ref{lem:(k,d)_no_a=1} that we only have to consider $a=2$.  
	
	By adding the edge $x_0x_2=xy$, we obtain a 5-regular graph with all internal faces of size three. By Theorem~\ref{thm:no-2-conn-1NP}, the outer face now must be also a triangle, as otherwise we would have a 2-connected 1-nearly Platonic graph with the outer face of size more than three. Therefore, the new graph is the icosahedron and the original one was the icosahedron without an edge  shown in Figure~\ref{fig:(5;4,4|3)}.
\end{proof}

\begin{figure}[H]
\centering
\includegraphics[scale=.14]{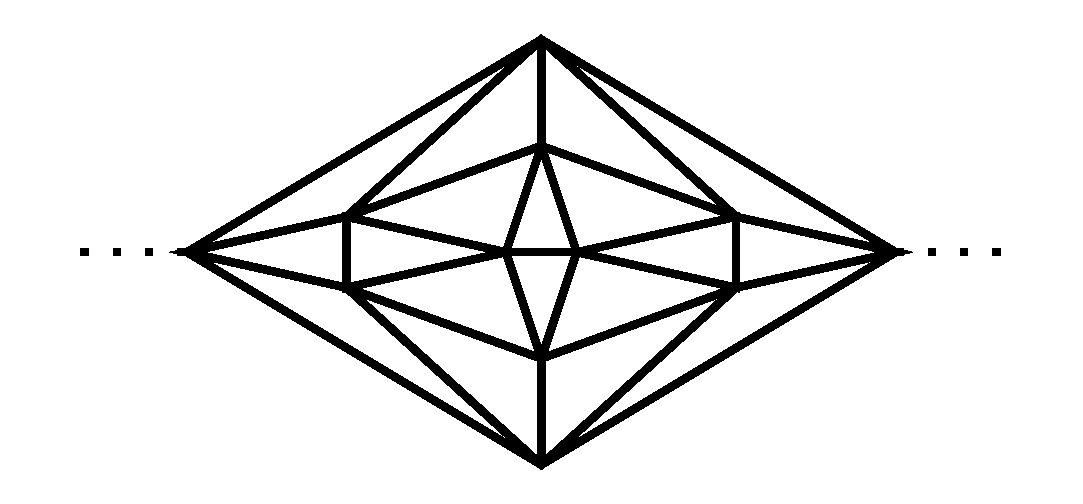}
\caption{Unique $(5;4,4|3)$-block}\label{fig:(5;4,4|3)}
\end{figure}

\begin{lemma}\label{lem:(5,3;4,3)-DNE}
	A $(5;4,3|3)$-block  does not exist.
\end{lemma}

\begin{proof}
	We have $1\leq a\leq 3$ by Proposition~\ref{prop:b-reduction}. But we cannot have $a=3$ by Lemma~\ref{lem:(k,d)_no_a=d}, or $a=2$ by Lemma~\ref{lem:(k,d)_no_a=d-1}. Hence, $a=1$. 
	
	We again create two copies of $B$ with $t$ inner triangular faces and assume $\deg(x)=\deg(x')=4$ and  $\deg(y)=\deg(y')=3$. Now we add a new vertex $z$ and edges $zx,zx',zy,zy',yy'$. This way we obtain a 2-connected graph with $2t+3$ inner triangular faces in which $\deg(z)=4$ and all other vertices are of degree five. This graph would be a $(5;4|3)$-endblock, which does not exist by Theorem~\ref{thm:no-endblocks}.
\end{proof}

\begin{lemma}\label{lem:(5,3;4,2)-DNE}
		A $(5;4,2|3)$-block  does not exist.
\end{lemma}

\begin{proof}
  We have $a=1$ by Proposition~\ref{prop:b-reduction} and Lemmas~\ref{lem:(k,d)_no_a=d} and~\ref{lem:(k,d)_no_a=d-1}.
  
 Suppose $\deg(x) = 2$. We remove $x$ and obtain a $(5;4,3|3)$-block. But by Proposition~\ref{lem:(5,3;4,3)-DNE}, it is impossible.
\end{proof}

Now we investigate $(5;3,d_2|3)$-blocks.

\begin{lemma}\label{lem:(5,3;3,3)-DNE}
	A $(5;3,3|3)$-block does not exist.
\end{lemma}

\begin{proof}
	By Lemma~\ref{lem:(5,3;3,d_2)-reduction} we have $1\leq a\leq 3$. If $a=1$, we create two copies of the block $B$ with $t$ inner triangular faces (and $b>1$) and amalgamate the edges $xy$ and $x'y'$. This creates a 2-connected 5-regular graph with $2t$ inner triangular faces and the outer face of size $2b\geq4$. Such graph would be 1-nearly Platonic and cannot exist by Theorem~\ref{thm:no-2-conn-1NP}.
	
	When $a=2$, then by adding the edge $xy=x_0x_2$ we would obtain a $(5;4,4|3)$-block whose non-existence was proved in Lemma~\ref{lem:(5,3;4,4)-uniq}.
	
	For $a=3$, we replace the edge $x_2x_3=x_2y$ by edge $x_0x_2$, creating a new triangular face. Now $\deg(x)=4,\deg(y)=2$ and the boundary path from $x$ to $y$ is $x=x_0,x_2,v,y$ for some $v$. If $v-x_j$ for some $j>3$, then the block bounded by $x_0,x_2,x_j,x_{j+1},\dots,x_{a+b-1}$ is a $(5;4,3|3)$-block or $(5;4,2|3)$-block where $v=x_j$ is of degree three or two, respectively. Such blocks do not exist by Lemmas~\ref{lem:(5,3;4,3)-DNE} and~\ref{lem:(5,3;4,2)-DNE}.
	
	If $v$ is an inner vertex of $B$, then we obtain a $(5;4,2|3)$-block with $a=3$, which cannot exist by Lemma~\ref{lem:(k,d)_no_a=d}. All cases have been covered and the proof is complete. 
\end{proof}

\begin{lemma}\label{lem:(5,3;3,2)-uniq}
	The $(5;3,2|3)$-block is unique.
\end{lemma}

\begin{proof}
	By Lemma~\ref{lem:(5,3;3,d_2)-reduction} we have $1\leq a\leq 3$. 
	If $a=1$, we create three copies $B^0,B^1,B^2$ of the block $B$ with $\deg(x^i)=3$ and $\deg(y^i)=2$ in each copy $B^i$. Assume that $B$ has $t$ internal triangular faces and observe that $b>1$ . Then we amalgamate $x^i$ with $y^{i+1}$ for all $i=0,1,2$, where the superscripts are calculated modulo 3. This way we obtain a 2-connected 5-regular graph with $3t+1$ inner triangular faces and the outer face of size $3b\geq6$. Because no such graph exists by Theorem~\ref{thm:no-2-conn-1NP}, this case is impossible.
	
	When $a=2$, then we add the edge $xy=x_0x_2$ and obtain a $(5;4,3|3)$-block, which cannot exist by Lemma~\ref{lem:(5,3;4,3)-DNE}.
	
	Now let $a=3$ and the outer face be $x_0,x_1,\dots,x_{a+b-1}$, where $\deg(x_0)=3, \deg(x_3)=2$, and $\deg(x_i)=5$ otherwise. 
	
	Amalgamate $x_0$ and $x_3$ into a new vertex $x'$ of degree five so that the new triangular face is $x',x_1,x_2$ and it is an inner face. The outer face is now $x',x_4,x_5,\dots,x_{a+b-1}$, and has size $a+b-3$.
	
	If $b>3$, we have $a+b-3\geq 4$. But then the resulting graph is a 1-nearly Platonic graph of type $(5|3)$ with exceptional face of size at least four, which is non-existent by Theorem~\ref{thm:no-2-conn-1NP}. Therefore, $b=3$. But then the new amalgamated graph has outer boundary of size three, namely $x',x_4,x_5$. Clearly, we have obtained the icosahedron, and the original block shown in Figure~\ref{fig:(5;3,2|3)} is unique.

\begin{figure}[H]
\centering
\includegraphics[scale=.14]{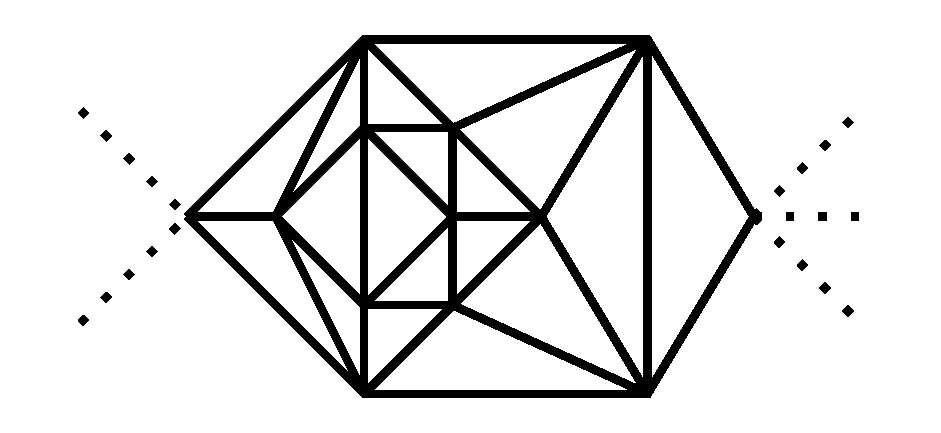}
\caption{Unique $(5;3,2|3)$-block}\label{fig:(5;3,2|3)}
\end{figure}
	
	Since there are no other values of $a$ to investigate, the proof is now complete.
\end{proof}

For the remaining case, that is, the $(5;2,2|3)$-block, we show directly the non-existence of a graph containing such block, rather than investigating the block on its own.

\begin{lemma}\label{lem:(5,3;2,2)-DNE}
	There is no $2$-nearly Platonic graph with touching faces containing a $(5;2,2|3)$-block.
\end{lemma}
		 
\begin{proof}
	Suppose such a graph with a $(5;2,2|3)$-block does exist.
	Let $x_0$ and $x_a$ be the two vertices of degree two. First we show that each of them must be shared by a $(5;3,d_2|3)$-block. Suppose it is not the case, such a block $B$ exists in a 2-nearly Platonic graph $G$ and $x_0$ does not belong to any $(5;3,d_2|3)$-block. Denote the neighbors of $x_0$ not in $B$ by $v_1,v_2,v_3$ and assume that there is a drawing of $G$ where the neighbors of $x_0$ in clockwise order are $x_1,v_1,v_2,v_3,x_{a+b-1}$. At least one of edges $x_0v_1,x_0v_3$ does not belong to any $(5;d_1,d_2|3)$-block, say it is $v_1$. Then the paths $P_1=x_1,x_0,v_1, P_2=v_1,x_0,v_2$, and $P_3=v_3,x_0,x_{a+b-1}$ all belong to boundaries of non-triangular faces. But because we only have two such faces,  $F_1$ and $F_2$, one of them must contain $x_0$ twice. Say it is $F_1$. Hence, $F_1$ is self-touching, and we must have an endblock of some kind, which is impossible by Theorem~\ref{thm:no-endblocks}.
	
	Hence, both $x_0$ and $x_a$ must belong to some $(5;3,d_2|3)$-block. Because neither $(5;4,3|3)$-block nor $(5;3,3|3)$-block exist, it must be a  $(5;3,2|3)$-block. It should be now obvious that if we attach such a block to each $x_0$ and $x_a$, the new graph will again have two new vertices of degree two, and such a chain of blocks can never be closed to create a 2-connected $2$-nearly Platonic graph. Therefore, the graph will have connectivity one, and consequently contain an endblock. This is impossible by Theorem~\ref{thm:no-endblocks} and the proof is complete.	
\end{proof}

\subsection{Classification: touching exceptional faces}
	\label{subs:touching-main-result}

Based on our lemmas, we can now state the main result of this section.

\begin{theorem}\label{thm:touching}
	There are exactly seven infinite families of $2$-nearly Platonic graphs with touching exceptional faces; one of each of types $(3|3),(3|4),(3|5)$, two of type $(4|3)$, and two of type $(5|3)$, shown in Figures~\ref{fig:chain(3,3)}--\ref{fig:chain(5,3)b}.
	
	Moreover, all these graphs have the two exceptional faces of the same size.
\end{theorem}

\begin{proof}
	For graphs of type $(3|d)$, the only blocks can be $(3;2,2|d)$ and they are all unique by Lemma~\ref{lem:(3,d;2,2)-uniq}.
	
	For type $(3|3)$ the only possible block is the $(3;2,2|3,\langle 2,2\rangle)$-block isomorphic to the tetrahedron with one removed edge, and the graph must be a chain  alternating the $(3;2,2|3,\langle 2,2\rangle)$-blocks and graphs $K_2$.

	\begin{figure}[H]
		\centering
		\includegraphics[scale=.2]{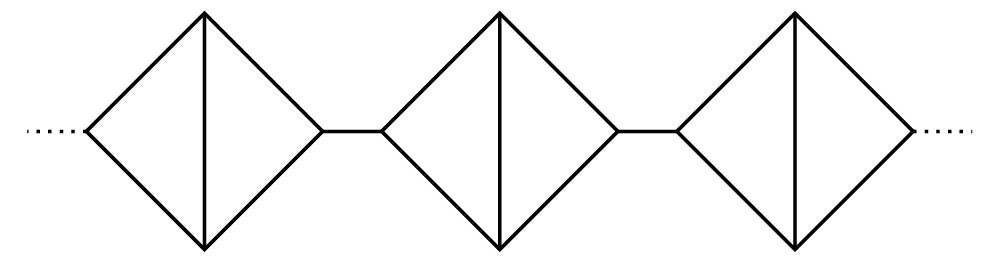}
		\caption{Chain of  blocks of type $(3;2,2|3)$ and $K_2$, tetrahedron edge cycle}\label{fig:chain(3,3)}
	\end{figure}
	
	Next, for type $(3|4)$ the only possible block is the $(3;2,2|4,\langle 3,3\rangle)$-block isomorphic to the cube with one removed edge, and the graph is a chain  alternating the $(3;2,2|4,\langle 3,3\rangle)$-blocks and graphs $K_2$.

	\begin{figure}[H]
		\centering
		\includegraphics[scale=.2]{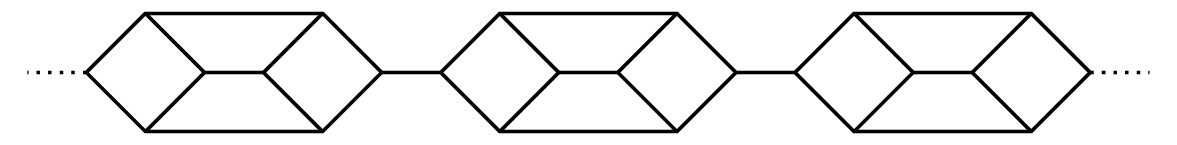}
		\caption{Chain of  blocks of type $(3;2,2|4)$ and $K_2$, cube edge cycle}\label{fig:chain(3,4)}
	\end{figure}
	
	Once again, for type $(3|5)$ the only possible block is the $(3;2,2|5,\langle 4,4\rangle)$-block isomorphic to the dodecahedron with one removed edge, and the graph is a chain  alternating the $(3;2,2|5,\langle 4,4\rangle)$-blocks and graphs $K_2$.

	\begin{figure}[H]
		\centering
		\includegraphics[scale=.25]{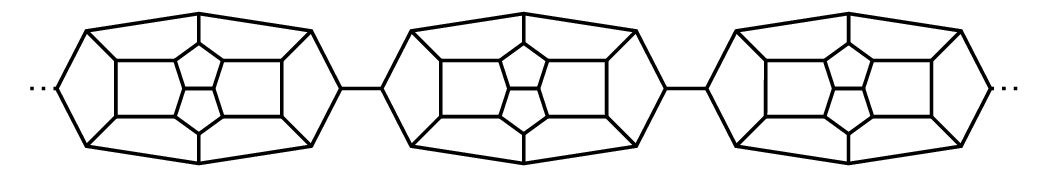}
		\caption{Chain of  blocks of type $(3;2,2|5)$ and $K_2$, dodecahedron edge cycle}\label{fig:chain(3,5)}
	\end{figure}

	For type $(4|3)$, the blocks could possibly be only of type $(4;3,3|3), (4;3,2|3)$, or $(4;2,2|3)$. A $(4;3,2|3)$-block does not exist by Lemma~\ref{lem:(4,3;3,2)-DNE}; the other two are unique by Lemmas~\ref{lem:(4,3;3,3)-uniq} and~\ref{lem:(4,3;2,2)-uniq}. Hence, the graph is either a chain consisting of the  $(4;3,3|5,\langle 2,2\rangle)$-blocks (that is, octahedrons without an edge) and graphs $K_2$, or a chain of $(4;2,2|5,\langle 3,3\rangle)$-blocks, arising from octahedron by splitting one vertex.

	\begin{figure}[H]
		\centering
		\includegraphics[scale=.25]{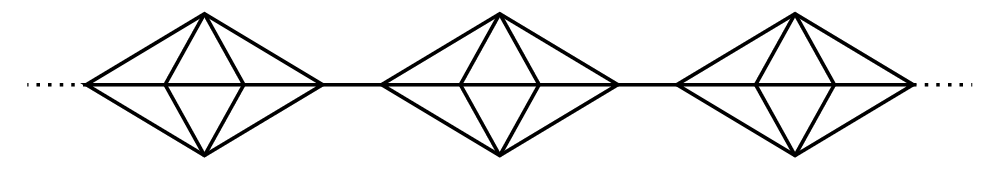}
		\caption{Chain of  blocks of type $(4;3,3|3)$  and $K_2$, octahedron edge cycle}\label{fig:chain(4,3)a}
	\end{figure}

	\begin{figure}[H]
		\centering
		\includegraphics[scale=.25]{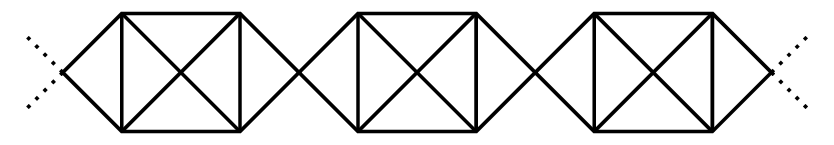}
		\caption{Chain of  blocks of type $(4;2,2|3)$, octahedron vertex cycle}\label{fig:chain(4,3)b}
	\end{figure}

	For type $(5|3)$, blocks of type $(5;4,3|3)$ and $(5;4,2|3)$ do not exist by Lemmas~\ref{lem:(5,3;4,3)-DNE} and~\ref{lem:(5,3;4,2)-DNE}, respectively. The block of type $(5;4,4|3)$ is unique by Lemma~\ref{lem:(5,3;4,4)-uniq}; it is the $(5;4,4|3, \langle 2,2\rangle)$-block, isomorphic to the icosahedron with one removed edge. The resulting graph then is a chain of the $(5;4,4|3, \langle 2,2\rangle)$-blocks alternating with graphs $K_2$.

	\begin{figure}[H]
		\centering
		\includegraphics[scale=.25]{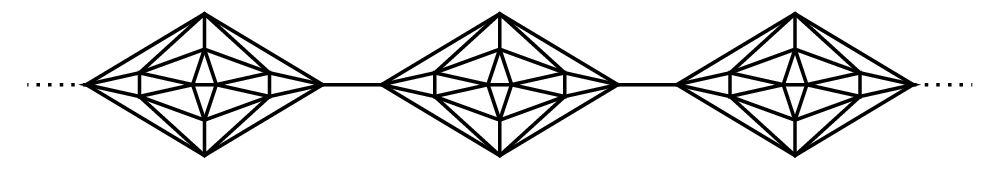}
			\caption{Chain of  blocks of type $(5;4,4|3)$  and $K_2$, icosahedron edge cycle}\label{fig:chain(5,3)a}
	\end{figure}
	
	Blocks of type $(5;3,3|3)$ and $(5;2,2|3)$ do not exist by Lemmas~\ref{lem:(5,3;3,3)-DNE} and~\ref{lem:(5,3;2,2)-DNE}. The block of type $(5;3,2|3)$ is unique by Lemma~\ref{lem:(5,3;3,2)-uniq}; it is the $(5;3,2|3, \langle 3,3\rangle)$-block, obtained from the icosahedron by splitting one vertex into two vertices of degree three and two, respectively. The resulting graph then is a chain of the $(5;3,2|3, \langle 2,2\rangle)$-blocks.

	\begin{figure}[H]
		\centering
		\includegraphics[scale=.3]{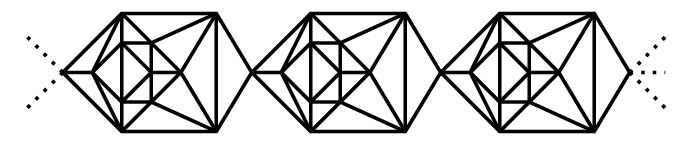}
			\caption{Chain of  blocks of type $(5;3,2|3)$, icosahedron vertex cycle}\label{fig:chain(5,3)b}
	\end{figure}
	
\end{proof}

%% file: non-touching_07e.tex

\section{Non-touching exceptional faces}\label{sec:non-touching}

\subsection{New notions}\label{subs:new-notions}

Let $F_1, F_2$ be the disjoint outer and inner disparate faces, respectively. We denote their respective boundaries by $x_1,x_2,\dots ,x_n$, and $y_1,y_2,\dots ,y_m $ in clockwise order. We define the distance between $F_1$ and $F_2$ as
 $$
	 \dist(F_1,F_2)=\min\{\dist(x_i,y_j)|x_i\in F_1,y_j\in F_2\}.
 $$

In a subgraph of a 2-nearly Platonic graph of type $(k|d)$, a vertex is \emph{saturated}, if it is of degree $k$. It should be obvious that in a 2-nearly Platonic graph with non-touching exceptional faces, each vertex must belong to at least two faces of size $d$. Similarly, in a subgraph of a 2-nearly Platonic graph of type $(k|d)$, a path of length $d-1$ is \emph{weakly saturated}, if all its internal vertices are of degree $k$. 

We start with some easy observations regarding the graphs of types $(3|3)$, $(3|4)$, and $(4|3)$.

\subsection{Graphs of types $(3|3)$, $(3|4)$, and $(4|3)$}\label{subs:small-types}

\begin{obs}\label{obs:no_(3|3)_non-touching}
	There is no $2$-nearly Platonic graph of type $(3|3)$ with non-touching exceptional faces.
\end{obs}

\begin{proof}
	By contradiction. Let $\dist(F_1,F_2)=1$. Then there is an edge $x_iy_i$ for some $i$. Because the faces are non-touching, we have $x_{a}\neq y_{b}$ for any $a,b$. Vertex $x_i$ is saturated, having neighbors $x_{i-1},x_{i+1},y_i$, and must belong to a triangular face $x_i,x_{i+1},y_i$. But then $y_i$ is of degree at least four, a contradiction. 
	
	If $\dist(F_1,F_2)=\dist(x_i,y_i)\geq2$, then we have a path $x_i,v_1,v_2,\dots,y_i$ (where possibly $v_2=y_j$). Again, $x_i$ is saturated, hence must belong to triangular faces $x_i,x_{i-1},v_1$ and $x_i,x_{i+1},v_1$, and $v_1$ must be of degree at least four, a contradiction again.
\end{proof}

\begin{obs}\label{obs:(3|4)_non-touching}
	The only $2$-nearly Platonic graph of type $(3|4)$ with non-touching exceptional faces is a prism.
\end{obs}

\begin{proof}
	If $\dist(F_1,F_2)=1$, the graph will be a prism. 
	Let the shortest path be $x_1y_1$, weakly saturating the path $x_2,x_1,y_1,y_2$, and since the common face is of degree four, $x_2y_2$ is forced. Using the same argument repeatedly, edge $x_iy_i$ is forced for every $i$. Hence, the graph must be a prism. 
	
	If $\dist(F_1,F_2)=\dist(x_1,y_1)\geq2$,  let $x_1,v_1,v_2,\dots ,y_1$ be the shortest path, where $v_2$ can be equal to $y_1$. Then $v_1$ will have one more neighbor, say $w_1$, WLOG in the clockwise direction. This saturates $v_1$ and thus $v_2$ and $x_n$ are adjacent. But now we have a shorter path $x_n,v_2,\dots,y_1$, a contradiction.
\end{proof}

\begin{obs}\label{obs:(4|3)_non-touching}
	The only $2$-nearly Platonic graph of type $(4|3)$ with non-touching exceptional faces is an antiprism.
\end{obs}

\begin{proof}
		If $\dist(F_1,F_2)=1$, the graph will be an anti-prism. Let $x_1y_1$ be a shortest path. Vertex $x_1$ has neighbors $x_2,x_n,y_1$ and some $v_1$, which can be placed WLOG so that the edge $x_1v_1$ is placed between edges $x_1x_n$ and $x_1y_1$.  Then $x_1$ is saturated, and we must have edge $x_2y_1$. Now $y_1$ is saturated, which forces edge $x_2y_2$. After repeating the argument $n$ times, we obtain an anti-prism.

		Now suppose that $\dist(F_1,F_2)=\dist(x_1,y_1)\geq2$,  and  $x_1,v_1,v_2,\dots ,y_1$ is the shortest path, where $v_2$ can again be equal to $y_1$.
		
		Let $w_1$ be the fourth neighbor of $x_1$ and WLOG suppose it is located  counter-clockwise from $x_1$. Now $x_1$ is saturated and $v_1$ and $x_2$ are adjacent. For the same reason, saturation of $x_1$, we must have the edge $w_1v_1$. Notice that $v_1$ is now saturated, which forces also the edge $x_2v_2$. This would mean that $\dist(x_2y_1)<\dist(x_,y_1)=\dist(F_1,F_2)$, which is a contradiction, and the proof is complete.
\end{proof}

\subsection{Graphs of type $(3|5)$}\label{subs:type-(3|5)}

For the $(3|5)$ case, we need several lemmas to determine the distance between the two non-touching exceptional faces.

\begin{lemma}\label{lem:ladder}
	Suppose $G$ is a $2$-nearly Platonic graph of type $(3|5)$ with non-touching exceptional faces $F_1, F_2$.
	
	Let $l=\dist(F_1,F_2)$ and suppose $l\geq 3$. Let $x_1,v_1,v_2,\dots,y_1$ be a path of length $l$. Denote by $w_i$ the third neighbor of $v_i$ for $i=1,2,\dots,l-1$, and assume $w_1$ is located clockwise from $v_1$. Then all vertices $w_{2j+1}$ are located clockwise from the path $x_1,v_1,v_2,\dots,y_1$, while all vertices $w_{2j}$ are located counter-clockwise from $x_1,v_1,v_2,\dots,y_1$.
\end{lemma}

\begin{proof}
	First observe that $w_2$ must be placed counter-clockwise from $v_2$. For if not, then the path $x_n,x_1,v_1,v_2,v_3$ is weakly saturated and forces edge $x_nv_3$. Then $\dist(x_ny_1)<\dist(x_,y_1)=\dist(F_1,F_2)$, which is a contradiction.

	Now let $i$ be the smallest subscript such that $w_i$ and $w_{i+1}$ are both placed in the same direction, say counter-clockwise from $v_i$ and $v_{i+1}$, respectively. Then the path $w_{i-1},v_{i-1},v_i,v_{i+1},v_{i+2}$ is weakly saturated, which forces edge $w_{i-1}v_{i+1}$. However,this creates a path $x_1,v_1,\dots,v_{i-1},w_{i-1},v_{i+2},v_{i+3},y_1$ of length $l-1<\dist(F_1,F_2)$, which is impossible. This contradiction completes the proof.
\end{proof}

\begin{lemma}\label{lem:1-or-3}
	Suppose $G$ is a $2$-nearly Platonic graph of type $(3|5)$ with non-touching exceptional faces $F_1, F_2$ and $\dist(F_1,F_2)=l$. Then $l=1$ or $l=3$.
\end{lemma}

\begin{proof}
  Let the shortest path be as in the previous proof, and third neighbors $w_i$ of $v_i$ be placed clockwise for odd subscripts, and counter-clockwise for even subscripts. We first want to show that the shortest path cannot have length more than three. 
  
  Suppose it does. Then $w_2$ is placed counter-clockwise, and $x_n,x_1,v_1,v_2,w_2$ is a weakly saturated path, forcing edge $x_nw_2$. Similarly, $w_4$ (which can be equal to $y_m$) is placed counter-clockwise, and $w_2,v_2,v_3,v_4,w_4$ is a weakly saturated path, forcing edge $w_2w_4$. This creates a  path  $x_n,w_2,w_4,v_4,v_5\dots,y_1$ of length at most $l-1$, a contradiction. 
  
  Now suppose $l=2$, and denote the shortest path between $F_1$ and $F_2$ by $x_1,v_1,y_1$. Again suppose the third neighbor $w_1$ of $v_1$ is placed clockwise from $v_1$. Then since the path $x_n,x_1,v_1,y_1,y_m$ is weakly saturated, we muse have $x_ny_m$ as an edge. Then $\dist(F_1,F_2)=\dist(x_n,y_m)=1$, which is impossible.
\end{proof}

So we just proved that the distance can only be one or three. In fact, the structures of the graphs for both cases are determined for both cases, which will be shown in the next lemma. Specifically, for the distance one case, we will shown that by some operations, every such graph could become a 2-nearly Platonic graphs with touching faces while the lengths of the exceptional faces do not change. And for the distance three case, we could reduce it to the smallest such graph and determine its structure.

\begin{lemma}\label{lem:(3|5)_l=1}
	There is exactly one infinite class of $2$-nearly Platonic graphs of type $(3|5)$ with non-touching exceptional faces $F_1, F_2$ and $\dist(F_1,F_2)=1$.
\end{lemma}

\begin{proof}
	For $l=1$, assume there is the edge $x_1y_1$. We can now split the edge into two edges $x'_1y'_1$ and $x''_1y''_1$. We create five copies of this graph, and amalgamate the edge $x''_1y''_1$ in the $i$-th copy with $x'_1y'_1$ in the $(i+1)$-st copy, with $i$ taken modulo 5. It should be clear that the new graph is still 2-nearly Platonic with two exceptional faces, one of size $5n$, and the other of size $5m$.
	
	Then we relabel the vertices of the exceptional faces. Denote one of the edges $x'_{1}y'_{1}$ by $w_1z_1$, and let our two exceptional faces be $w_1w_2\dots w_{5n}$ and $z_1z_2\dots z_{5m}$. We add edges $w_1w_5$, $w_6w_{10}, \dots, w_{5n-4}w_1$ and remove edges $w_nw_1, w_5w_6, \dots,$ $w_{5n-5}w_{5n-4}$. This way we obtain a 2-nearly Platonic graph with touching faces since the two new faces share the vertex $z_1$. Note that this operation does not change the size of the exceptional faces. 
	
	As we proved before, the conjecture is true for the touching exceptional faces case, thus we can conclude that $5m = 5n$, or $m=n$. Since we have classified the structure of the 2-nearly Platonic graphs for the touching exceptional faces, we can determine the structure of all the non-touching case of type $(3|5)$ by reversing the operations. The fundamental block of this type is shown in Figure~\ref{fig:(3|5)-l=1}.
\end{proof}

\begin{figure}[H]
	\centering
	\includegraphics[scale=.5]{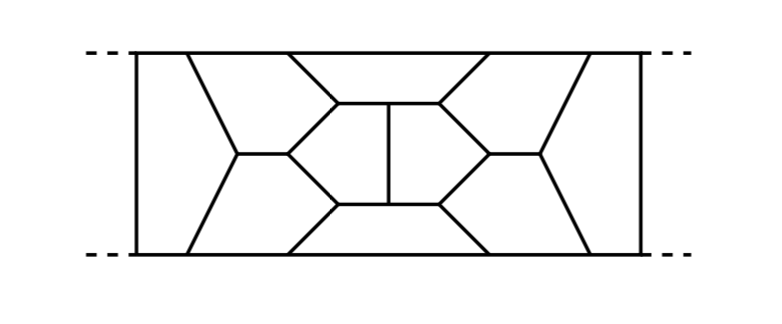}
	\caption{Fundamental block with non-touching faces of type $(3|5)$ with $l=1$}\label{fig:(3|5)-l=1}
\end{figure}

\begin{lemma}\label{lem:(3|5)_l=3}
	There is exactly one infinite class of $2$-nearly Platonic graphs of type $(3|5)$ with non-touching exceptional faces $F_1, F_2$ and $\dist(F_1,F_2)=3$.
\end{lemma}

\begin{proof}
	Recall that the exceptional faces $F_1$ and $F_2$ are bounded by cycles $x_1,x_2,$ $\dots,x_n$ and $y_1,y_2,\dots,y_m$, respectively.
	We have $\dist(F_1,F_2)=3$ and denote a shortest path by $x_1,v_1,v_2,y_1$ and  the third neighbors of $v_1, v_2$ by $w_1, w_2$, respectively. WLOG suppose $w_1$ is placed clockwise from the path $x_1,v_1,v_2,y_1$, then by Lemma \ref{lem:ladder},  $w_2$ is placed counter-clockwise. Since $x_1, v_1, v_2, y_1$ are all saturated, thus $w_1y_2$ and $w_2x_n$ must be adjacent. Also, the third neighbor of $w_1$, call it $z_1$, must be placed clockwise from the path $x_1,v_1,w_1,y_2$ while the third neighbor $z_2$ of $w_2$ must be counter-clockwise from. $x_n,w_2,v_2,y_1$
	
	Now consider the path $x_1,v_1,w_1,y_2$ and $x_n,w_2,v_2,y_1$. They have the same structure in the sense that the third neighbors of both corresponding vertices $v_1$ and $w_2$ are placed counter-clockwise from the respective paths, while the third neighbors of the corresponding vertices $w_1$ and $v_2$ are placed clockwise from the respective paths. Therefore, we can amalgamate those two paths together (omitting the loop arising from edge $v_1v_2$) and reduce the size of the two exceptional faces by one. More precisely, we amalgamate $x_n$ with $x_1$, $w_2$ with $v_1$, $v_2$ with $w_1$, and $y_1$ with $y_2$.
  
	Since the reduced graph also maintains the property that the distance of two exceptional faces is three, we can repeat this procedure. We can also reverse it by taking a path $x_1,v_1,v_2,y_1$ and splitting it into two paths, $x'_1,v'_1,v'_2,y'_1$ and $x''_1,v''_1,v''_2,y''_1$ so that the original edge $w_2v_2$ becomes $w_2v'_2$ and $v_1w_1$ becomes $v''_1w_1$. Now we add edges $x'_1x''_1, v'_1v''_2$ and $y'_1y''_1$ to obtain a new graph in which the structure is preserved while the exceptional faces have sizes $n+1$ and $m+1$, respectively.

	Now we want to show that $m=n$. We assume WLOG that $n\leq m$. If $n=5+t$ for some $t\geq0$, we reduce the graph $t$ times to obtain $n'=5$ and $m'>5$, and we have a 1-nearly Platonic graph of type $(3|5)$ with the exceptional face of size $m'>5$, which is impossible by Theorem~\ref{thm:no-2-conn-1NP}. Hence, we must have $m=n$.
	
	If $n<5$, say $n=5-s$, we expand the graph $s$ times, and get $n'=5$ and $m'>5$. By the same argument as above, such graph cannot exist. Therefore, we must have $m=n$.
	
	The fundamental block of this type is shown in Figure~\ref{fig:(3|5)-l=3}.
\end{proof}

\begin{figure}[H]
	\centering
	\includegraphics[scale=.5]{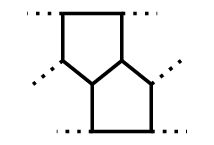}
	\caption{Fundamental block with non-touching faces of type $(3|5)$ with $l=3$}\label{fig:(3|5)-l=3}
\end{figure}

Hence in either case, we know the structure of the graph and the two exceptional faces have the same degree for each case.

\subsection{Graphs of type $(5|3)$}\label{subs:type-(5|3)}

For the $(5|3)$ case, we also first discuss the distance between the two exceptional faces.

We start by showing that the number of vertices on $F_1$ and $F_1$ is half of the order of the graph. Denote the order of the graph by $|V|$, and we know there are $m$ and $n$ vertices on $F_1$ and $F_2$, respectively. Since the graph is 5-regular, there are $5|V|/2$ edges. Also, we can count the number of edges using the number of faces. By Euler's formula, the number of faces, denoted by $|F|$, is $|E|-|V|+2$, which is $5|V|/2-|V|+2$, or $3|V|/2+2$. Since all faces except two are triangles, and the other two faces are of degree $m$ and $n$, respectively, we have $3(|F|-2)+m+n=2|E|$. Because $|F|=3|V|/2+2$ and $|E|=5|V|/2$, we have $9|V|/2+m+n=5|V|$, or $2(m+n)=|V|$, as desired.

We summarize these findings as follows.

\begin{obs}\label{obs:|V|}
	Let $G$ with vertex set $V$ be a $2$-nearly Platonic graph with non-touching exceptional faces of sizes $m$ and $n$, respectively. Then $|V|=2(m+n)$.
\end{obs}

Now we use the fact that $|V|=2(m+n)$ to show the distance between $F_1$ and $F_2$ cannot be greater than two.

\begin{lemma}\label{lem:(5|3)_dist=1-or-2}
	Let $G$ be a $2$-nearly Platonic graph of type $(5|3)$ with non-touching exceptional faces $F_1$ and $F_2$ and  $\dist(F_1,F_2)=l$. Then $1\leq l\leq2$.
\end{lemma}

\begin{proof}
  Suppose the distance is two or more, and define the neighborhoods of $F_1, F_2$ as 
  $$
  N(F_j) = \{u\;|\; { ux\in E(G) \ \text{for some} \ x\in F_j}  \ \text{and} \ u\not\in F_j\}.
  $$
  We observe that $u\not\in F_{j+1}$ for $u\in N(F_j)$ as the distance is more than one. We want to show that $|N(F_j)|=2|F_j|$.   
  
  Let the $n$-cycle $x_1,x_2,\dots,x_n$ be the boundary of $F_1$ and $u_i^0, u_i^1, u_i^2 \in N(F_1)$ be the three neighbors of $x_i$, placed in that order. Since the faces inside the boundary are all triangles, $u_i^2$ and $u_{i+1}^0$ must be the same vertex. Also, $u_i^0$ and $u_i^1$ are forced to be adjacent, as well as $u_i^1$ and $u_i^2$. So in $N(F_1)$, we would have $n$ distinct vertices that have exactly two neighbors in $F_1$ each, and $n$ distinct vertices with exactly one neighbor in $F_1$ each. Thus, together there are $2n$ vertices.
  
  By applying the same argument to $N(F_2)$, we have $|N(F_2)|=2m$. Because we have $|F_1|=n, |F_2|=m$ and by Observation~\ref{obs:|V|}   $|V|=2(m+n)$,  it follows that
  $$
  |N(F_1)\cup N(F_2)| \leq n+m.
  $$
  But we also have 
  $$
  |N(F_1)| =2n \ \text{and} \ |N(F_2)|=2m.
  $$
  Indeed, $N(F_j)\subseteq N(F_1)\cup N(F_2)$, and 
  $$
  2n=|N(F_1)|\leq|N(F_1)\cup N(F_2)| \leq n+m,
  $$
  which yields $n\leq m$. By symmetry, looking at $N(F_2)$, we obtain $m\leq n$, which implies $m=n$, and consequently
  $$
  |F_1| = |F_2| =n \ \text{and} \  |N(F_1)\cup N(F_2)| =2n.
  $$
  Therefore,
  $$
   |F_1| + |F_2|  + |N(F_1)\cup N(F_2)| =4n =|V|
  $$
  and there is no vertex between the boundaries of $F_1$ and $F_2$ which would not be adjacent to vertices in both $F_1$ and $F_2$. Thus the distance between $F_1$ and $F_2$ cannot exceed two.
\end{proof}

In the previous lemma, we not only proved the statement, but we observed that if the distance is two, the conjecture holds. Moreover we can determine the structure in this case. By the proof of the lemma, all vertices other than the boundary of $F_2$ are at distance one from both $F_1$ and $F_2$. Let the vertices having one neighbor on $F_1$ be $v_i$ for $i=1,2,\dots,n$ and $x_iv_i$ be the edges. Let the common neighbor of $x_i$ and $x_{i+1}$ be $w_i$. Then there is the inner cycle $C_{2n}=v_1,w_1,v_2,w_2,\dots,v_n,w_n$ closing the triangles.

By symmetry, all vertices except the boundary of $F_1$ are at distance one  from $F_2$. Thus they are all in the set $\{v_1,v_2,\dots,v_n\}\cup\{w_1,w_2,\dots,w_n\}$. Clearly, for $i=1,2,\dots,n$  each $w_i$ is already of degree four and must be a neighbor of exactly one vertex on $F_2$, say $y_j$. This forces $v_i$ to be the remaining neighbor of both $y_{i-1}$ and $y_{i}$.  This uniquely determines the structure of the graph. We summarize our findings in the following lemma.

\begin{lemma}\label{lem:(5|3)_dist=2}
	The class of $2$-nearly Platonic graphs of type $(5|3)$ with non-touching exceptional faces $F_1$ and $F_2$ and  $\dist(F_1,F_2)=2$ is unique.
\end{lemma}

\begin{proof}
	The proof was given above, and the stucture can be seen in  {Figure~\ref{fig:(5|3)-l=2} below}.
\end{proof}

\begin{figure}[H]
	\centering
	\includegraphics[scale=.4]{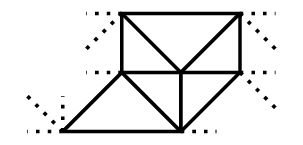}
	\caption{Fundamental block with non-touching faces of type $(5|3)$ with $l=2$}\label{fig:(5|3)-l=2}

\end{figure}
There is only one case left now, namely when $\dist(F_1,F_2)=1$. The method we will use  is similar to what we did in distance one case for type $(3|5)$. That is, we will use the results for the touching case  to prove the non-touching case.

\begin{lemma}\label{lem:(5|3)_dist=1}
	There are exactly two classes of $2$-nearly Platonic graphs of type $(5|3)$ with non-touching exceptional faces $F_1$ and $F_2$ and  $\dist(F_1,F_2)=1$.
\end{lemma}

\begin{proof}
    Because $\dist(F_1,F_2)=1$, we must have an edge joining the two faces, say $x_1y_1$. Up to symmetry, there are three possible structures for the neighbors of $x_1$ and $y_1$. The first case is that the two neighbors of $x_1$ inside the boundary are located clockwise from $x_1y_1$ while two neighbors of $y_1$ are counter-clockwise from $x_1y_1$. The second case is that only  $x_1$ has its two internal neighbors on  the same side of $x_1y_1$, say clockwise for it, and $y_1$ has neighbors on both sides of $x_1y_1$. The last case is that the two neighbors of $x_1$ inside the boundary are on different sides of $x_1y_1$, and so are the two neighbors of $y_1$. The reason why the four neighbors cannot be on the same side of $x_1y_1$, say clockwise from $x_1y_1$, is that if so, then the path $x_n,x_1,y_1$ is weakly saturated and we would need the edge $x_ny_1$ to complete the triangular face. However, this would make $y_1$ of degree six, which is impossible. The three possible cases are shown in Figure~\ref{3-cases} below.

    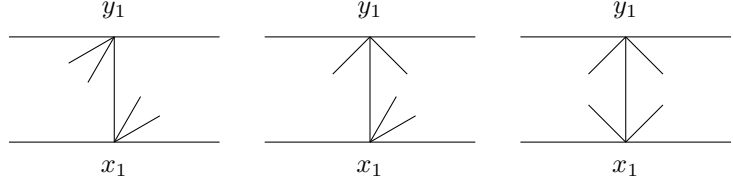
\begin{figure}[h]
        \begin{center}
            \begin{tikzpicture}[scale=.7]
                \node (a1) at (2,0) [label=below:$x_1$]{};
                \node (y1) at (2,2) [label=:$y_1$]{};
                \draw (0,0) -- (4,0);
                \draw (0,2) -- (4,2);
                \draw (2,0) -- (2,2);
                \draw (2,0) --++ (60:1);
                \draw (2,0) --++ (30:1);
                \draw (2,2) --++ (-120:1);
                \draw (2,2) --++ (-150:1);
            \end{tikzpicture}
            \hspace{1em}
            \begin{tikzpicture}[scale=.7]
                \node (a1) at (2,0) [label=below:$x_1$]{};
                \node (y1) at (2,2) [label=:$y_1$]{};
                \draw (0,0) -- (4,0);
                \draw (0,2) -- (4,2);
                \draw (2,0) -- (2,2);
                \draw (2,0) --++ (60:1);
                \draw (2,0) --++ (30:1);
                \draw (2,2) --++ (-45:1);
                \draw (2,2) --++ (-135:1);
            \end{tikzpicture}
            \hspace{1em}
            \begin{tikzpicture}[scale=.7]
                \node (a1) at (2,0) [label=below:$x_1$]{};
                \node (y1) at (2,2) [label=:$y_1$]{};
                \draw (0,0) -- (4,0);
                \draw (0,2) -- (4,2);
                \draw (2,0) -- (2,2);
                \draw (2,0) --++ (45:1);
                \draw (2,0) --++ (135:1);
                \draw (2,2) --++ (-45:1);
                \draw (2,2) --++ (-135:1);
            \end{tikzpicture}
            \caption{From left to right: case 1, 2, and 3}
            \label{3-cases}
        \end{center}
    \end{figure}


In fact, if we have the structure as described in the second case, we  obtain the same structure as in case one. We have the two neighbors of $x_1$ located clockwise from $x_1y_1$, which implies that the path $x_nx_1y_1$ is weakly saturated, and $x_ny_1$ must be an edge. Now, $x_n,x_1,y_1$ form a triangle and because both $x_1$ and $y_1$ are already saturated, $x_n$ cannot have a neighbor inside the triangle. Thus both remaining internal neighbors of $x_n$ are located counter-clockwise from $x_ny_1$, and  the two remaining neighbors of $y_1$ (one of which is $x_1$) are both on the clockwise side of $x_ny_1$, which is what we have in case one up to symmetry.

We reduced the problem to two cases, and will discuss them now one by one. For the first case where the two neighbors of $x_1$ inside the boundary are located clockwise from $x_1y_1$ while two neighbors of $y_1$ are counter-clockwise from $x_1y_1$, we can split the edge $x_1y_1$ to obtain a strip. Then we make three copies of the strip and attach them together, for the vertices that are incident with the splitting edge are symmetric. This way we obtain a larger 2-nearly Platonic graph with exceptional faces of degrees $3n$ and $3m$. We label the vertices again so that $x_1y_1$ is a path  from $F_1$ to $F_2$ and $x_1$ has two neighbors clockwise from $x_1y_1$. So $x_1y_2$ will be an edge connecting $F_1$ to $F_2$ as well. Then we remove edges $x_{3n}x_1, x_3x_4, \dots, x_{3n-3}x_{3n}$ and add edges $x_1x_3, x_4x_5, \dots, x_{3n-2}x_n$. The graph remains 5-regular and all but the two exceptional faces are triangles. Also, the long faces will have the same length as the graph before the operation. However, now the two exceptional faces share the vertex $y_1$, so by the previous result, the two faces must have the same size, i.e. $3m=3n$, thus $m=n$, as desired.

For the third case, where the two neighbors of $x_1$ inside the boundary are on different sides of $x_1y_1$, and the same holds for $y_1$, we can also split the edge $x_1y_1$, make three copies and glue them together. Then instead of adding and removing edges on only one of the exceptional faces as we did in the first case, we will add and remove edges on both inner and outer face. Again after the operation the two new exceptional faces share a vertex, which is one of the common neighbors of $x_1$ and $y_1$. Since the operation does not change the face size, we could conclude that $3m=3n$, and so $m=n$.

Since the class of 2-nearly Platonic graphs with touching faces obtained by these operations is unique as described in Lemma~\ref{lem:(5,3;3,2)-uniq}, it should be obvious that starting with graphs in Figure~\ref{3-cases} and reversing the steps, we obtain graphs in ~\ref{fig:(5|3)-l=1_ab} and~\ref{fig:(5|3)-l=1_c},  respectively.
\end{proof}

\begin{figure}[H]
    \centering
    \includegraphics[scale=.7]{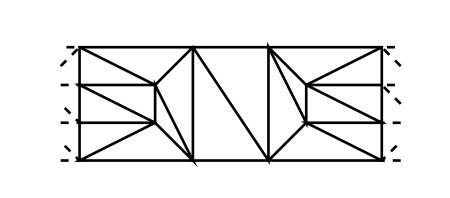}
    \caption{Case 1 and 2}
    \label{fig:(5|3)-l=1_ab}
\end{figure}

\begin{figure}[H]
    \centering
    \includegraphics[scale=.7]{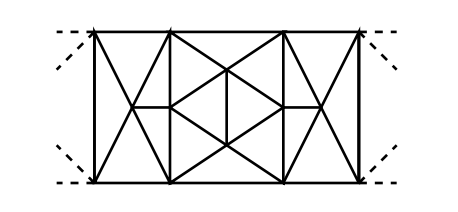}
    \caption{Case 3}
    \label{fig:(5|3)-l=1_c}
\end{figure}

\subsection{Classification: non-touching exceptional faces}
\label{subs:non-touching-main-result}

\begin{theorem}\label{thm:non-touching}
	There are exactly seven infinite families of $2$-nearly Platonic graphs with non-touching exceptional faces; the prism of type $(3|4)$ in Figure~\ref{fig:prism}, wo graphs of type $(3|5)$ in Figures~\ref{fig:barrel} and~\ref{fig:dodec}, antiprism of type $(4|3)$ in Figure~\ref{fig:antiprism}, tand three of type $(5|3)$ in Figures~\ref{fig:ico-wide},~\ref{fig:ico-1st} and~\ref{fig:ico-2nd}. 
	Moreover, all these graphs have the two exceptional faces of the same size.
\end{theorem} 

\begin{proof}
	The non-existence of 2-nearly Platonic graphs follows from Observation~\ref{obs:no_(3|3)_non-touching}. 
	The result for type $(3|4)$  follows from Observation~\ref{obs:(3|4)_non-touching}.

	\begin{figure}[H]
		\centering
		\includegraphics[scale=.3]{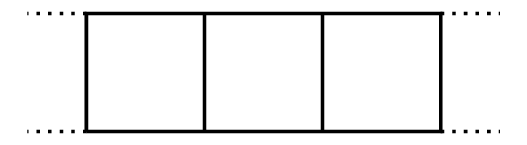}
		\caption{Prism, type $(3|4)$}\label{fig:prism}
	\end{figure}

	For type $(3|5)$ the result follows from Lemmas~\ref{lem:(3|5)_l=1} and~\ref{lem:(3|5)_l=3}.

	\begin{figure}[H]
		\centering
		\includegraphics[scale=.4]{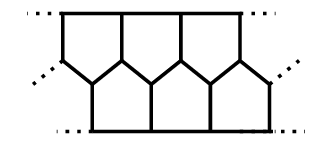}
		\caption{Barrel, type $(4|3)$}\label{fig:barrel}				
	\end{figure}

	\begin{figure}[H]
		\centering
		\includegraphics[scale=.25]{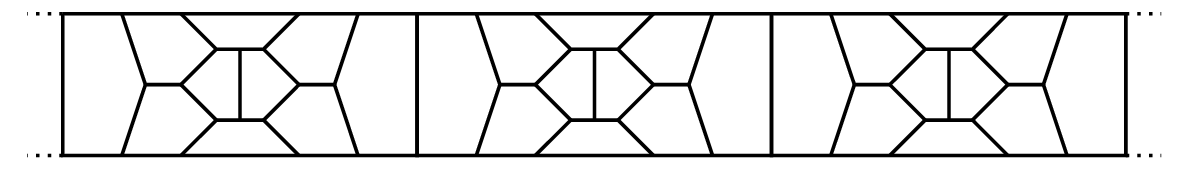} 
		\caption{Dodecahedron  thick cycle, type $(4|3)$}\label{fig:dodec}						
	\end{figure}
	
	The result for type $(4|3)$  follows from Observation~\ref{obs:(4|3)_non-touching}.

	\begin{figure}[H]
		\centering
		\includegraphics[scale=.3]{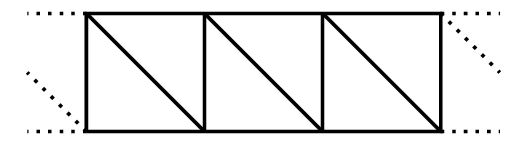}
		\caption{Antiprism, type $(4|3)$}\label{fig:antiprism}		
	\end{figure}

	Finally, for type $(5|3)$ the result follows from Lemmas~\ref{lem:(5|3)_dist=2}  and~\ref{lem:(5|3)_dist=1}.
\end{proof}	

	\begin{figure}[H]
		\centering
		\includegraphics[scale=.3]{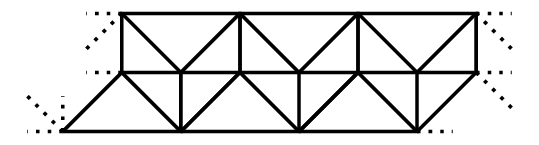}
		\caption{Icosahedron wide cycle, type $(5|3)$}\label{fig:ico-wide}						
	\end{figure}

	\begin{figure}[H]
		\centering
		\includegraphics[scale=.4]{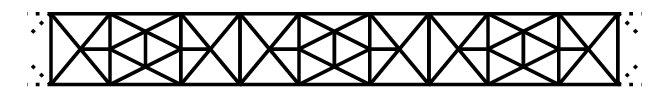}
		\caption{Icosahedron first thick cycle, type $(5|3)$}\label{fig:ico-1st}						
	\end{figure}

	\begin{figure}[H]
		\centering
		\includegraphics[scale=.45]{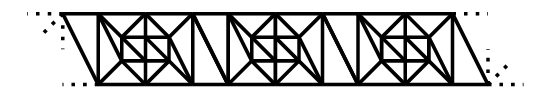}
		\caption{Icosahedron second thick cycle, type $(5|3)$}\label{fig:ico-2nd}						
	\end{figure}

%% file: conclusion_07e.tex

\section{Conclusion}\label{sec:conclusion}

We summarize our results in the form answering in the affirmative the conjecture by Keith, Froncek, and Kreher~\cite{KFK1}. For the respective classes of graphs, we slightly modified the terminology introduced in~\cite{KFK1}.

\begin{theorem}\label{thm:main5}
	All $2$-nearly Platonic graphs are balanced and belong to one of the following $14$ families, listed by type $(k|d)$:
	
	\begin{itemize}
		\item Type $(3|3)$: 
			\begin{enumerate}
				\item 	tetrahedron edge cycle (Figure~\ref{fig:chain(3,3)})
			\end{enumerate}
	
		\item Type $(3|4)$: 
			\begin{enumerate}
	 			\item[2.] cube edge cycle (Figure~\ref{fig:chain(3,4)})
	 			\item[3.] prism (Figure~\ref{fig:prism})
			\end{enumerate}

		\item Type $(3|5)$: 
			\begin{enumerate}
				\item[4.] dodecahedron edge cycle (Figure~\ref{fig:chain(3,5)})
				\item[5.] barrel (Figure~\ref{fig:barrel})
				\item[6.] dodecahedron thick cycle		(Figure~\ref{fig:dodec})		
			\end{enumerate}
				
		\item Type $(4|3)$: 
			\begin{enumerate}
				\item[7.] octahedron edge cycle (Figure~\ref{fig:chain(4,3)a})
				\item[8.] octahedron vertex cycle (Figure~\ref{fig:chain(4,3)b})
				\item[9.] antiprism			(Figure~\ref{fig:antiprism})
			\end{enumerate}

		\item Type $(5|3)$: 		
			\begin{enumerate}
				\item[10.]  icosahedron edge cycle (Figure~\ref{fig:chain(5,3)a})
				\item[11.]  icosahedron vertex cycle (Figure~\ref{fig:chain(5,3)b}) 
				\item[12.]  icosahedron wide cycle (Figure~\ref{fig:ico-wide}) 
				\item[13.]  icosahedron first thick cycle (Figure~\ref{fig:ico-1st}) 
				\item[14.]  icosahedron second thick cycle (Figure~\ref{fig:ico-2nd}) 
			\end{enumerate}
		
	\end{itemize} 
\end{theorem}

\begin{rem*}
	\emph{This paper was originally written as two independent papers by two disjoint pairs of co-authors. The methods used in both papers were very similar, and the papers differed mainly  in the structure of theorems and proofs. After we had learned about each other, we decided to join forces and merge the papers into one, which is presented here.
	}
\end{rem*}
